\documentclass[leqno]{elsarticle}

\usepackage{amsthm}
\newtheorem{theorem}{Theorem}
\newtheorem{lemma}[theorem]{Lemma}
\newtheorem{corollary}[theorem]{Corollary}
\newtheorem{definition}[theorem]{Definition}

\usepackage[english]{babel}
\usepackage{amsmath,amssymb,amsfonts,graphicx,subfigure}
\usepackage{hyperref}
\usepackage[ruled]{algorithm2e} % ubuntu package: texlive-science 
\newcommand{\diag}{DUMMY}  % renewed in myshortcuts.sty
\usepackage{myshortcuts}

\DeclarePairedDelimiter{\abs}{\lvert}{\rvert}

\usepackage{pgfplots}
\usetikzlibrary{external}
\usetikzlibrary{decorations.pathreplacing}
\usetikzlibrary{positioning,calc}

\makeatletter
\let\real=\pgfmath@calc@real
\let\minof=\pgfmath@calc@minof
\let\maxof=\pgfmath@calc@maxof
\let\ratio=\pgfmath@calc@ratio
\let\widthof=\pgfmath@calc@widthof
\let\heightof=\pgfmath@calc@heightof
\let\depthof=\pgfmath@calc@depthof
\makeatother

\iftrue % set to iftrue if you want to regen PDF-files in gfx_tmp. 
  \usepackage{tikz}
  \usepackage{pgfplots}
  \pgfplotsset{
    compat=newest,
    tick label style={font=\scriptsize},
    label style={font=\scriptsize},
    legend style={font=\scriptsize}
  }
  \usetikzlibrary{decorations.pathreplacing}
  \iffalse % Set to iftrue if you want to regen fig PDF-files. 
     \usepgfplotslibrary{external} % requires a recent version of pgf package
  \else
     \renewcommand{\tikzsetnextfilename}[1]{}
  \fi
\else
  \usepackage{tikzexternal}
  \tikzsetexternalprefix{gfx_tmp/}
  
\fi
\begin{document}
\begin{frontmatter}
\title{Iterative methods for the delay Lyapunov equation
with T-Sylvester preconditioning}
\author[EJ]{Elias Jarlebring}
\author[FP]{Federico Poloni}
\address[EJ]{Department of Mathematics, Royal Institute of Technology (KTH), Stockholm, SeRC Swedish e-Science Research Center. \texttt{eliasj@kth.se}}
\address[FP]{Department of Computer Science, University of Pisa, Italy. \texttt{federico.poloni@unipi.it}}
%\author{Elias Jarlebring\thanks{Department of Mathematics, Royal Institute of Technology (KTH), Stockholm. \texttt{eliasj@kth.se}.} \and Federico Poloni\thanks{Department of Computer Science, University of Pisa, Italy. \texttt{federico.poloni@unipi.it}.}} %Federico was here. Standard alphabetical order.
\date{\today}
\begin{abstract}
The delay Lyapunov equation is an important
matrix boundary-value problem which arises as an
analogue of the Lyapunov equation in the study of 
time-delay systems $\dot{x}(t) = A_0x(t)+A_1x(t-\tau)+B_0u(t)$.
We propose a new algorithm for the solution of the delay Lyapunov equation. Our method is based on the fact that the
delay Lyapunov equation can be
expressed as a linear system of equations, 
whose unknown is the value $U(\tau/2)\in\mathbb{R}^{n\times n}$, i.e., the delay Lyapunov matrix at time $\tau/2$.
This linear matrix equation with $n^2$ unknowns is solved
by adapting a preconditioned iterative method such as GMRES. The action of the $n^2\times n^2$ matrix associated to this linear system can be computed by solving a coupled matrix initial-value problem. 
A preconditioner for the iterative method is
proposed based on solving a T-Sylvester equation $MX+X^TN=C$,
for which there are methods available in the literature. 
We prove that the preconditioner is effective under
 certain assumptions.
%with a direct method. 
The efficiency of the approach is 
illustrated by applying it to a time-delay system stemming
 from the discretization of a
 partial differential equation with delay.
%when the problem is in a sense close to delay-free (when $A_1$ is small).
Approximate solutions to this problem can be obtained for problems of size up to $n\approx 1000$, i.e., a linear system with $n^2\approx 10^6$ unknowns,
a dimension which is outside of the capabilities of the other existing methods for the delay Lyapunov equation.
\end{abstract}
\end{frontmatter}
\section{Introduction}\label{sec:intro}
Consider the linear single-delay time-delay system defined
by  the equations
\begin{subequations}\label{eq:tds}%
\begin{eqnarray}%
   \dot{x}(t)&=&A_0x(t)+A_1x(t-\tau)+B_0u(t)\label{eq:tds.a}\\
   y(t)&=&C_0x(t),\label{eq:tds.b}
\end{eqnarray}%
\end{subequations}%
where $A_0,A_1\in\RR^{n\times n}$, $B_0\in\RR^{n \times m}$,
$C_0^\TT\in\RR^{n \times p}$. The general equation
\eqref{eq:tds} appears in many different fields. 
It is considered a very important 
topic in the field of systems and control, mostly due
to the fact that most feedback systems are 
non-instantaneous in the sense that there is a
delay between the observation (of for instance  the state)
and the action of the feedback. See monographs
\cite{Michiels:2007:STABILITYBOOK,Gu:2003:STABILITY}
and survey paper \cite{Richard:2003:SUMMARY} for 
literature on time-delay systems.

The delay Lyapunov equations associated
with \eqref{eq:tds} correspond to the problem of
finding $U\in \mathcal{C}^0([-\tau,\tau],\CC^{n\times n})$ such that
\begin{subequations}\label{eq:dlyaps}
\begin{eqnarray}
  U'(t)&=&U(t)A_0+U(t-\tau)A_1,\; t > 0, \label{eq:dlyapd_de}\\
  U(-t)&=&U(t)^\TT,\label{eq:dlyaps_sym}\\
  -W&=&  U(0)A_0+A_0^\TT U(0)+ U(\tau)^\TT A_1+A_1^\TT U(\tau), \label{eq:dlyaps_alg}
\end{eqnarray}
\end{subequations}
hold for a given a cost matrix $W=W^T\in\mathbb{R}^{n\times n}$ (in some applications, for instance, $W=C_0^TC_0$).

Equation \eqref{eq:dlyapd_de} is a matrix delay-differential equation
and   \eqref{eq:dlyaps_alg} is an algebraic condition
 involving $U(0)$, $U(\tau)$ and
$U(-\tau)=U(\tau)^T$ such that \eqref{eq:dlyaps} can be
interpreted as a  matrix boundary value problem.
In this paper we propose a new procedure to solve~\eqref{eq:dlyaps},
with the goal to have good performance for large $n$ ($n\approx 500-1000$, for instance).

%where, e.g., $W=C_0^\TT C_0$. 
%They are studied in notation in \cite{Kharitonov:2006:LYAPUNOVPLISCHKE}.

The delay Lyapunov equation generalizes the standard Lyapunov equation, since,
e.g., if we set $\tau=0$ the equation reduces to the standard Lyapunov
equation.  
Moreover, as established by the last decades of research, 
the delay Lyapunov equation is in many ways 
playing the same important role
for time-delay systems as the standard Lyapunov
equation plays for standard (delay free) linear time-invariant dynamical systems.
More precisely, the delay Lyapunov equation has been studied 
in the following ways. It has been extensively used
to characterize stability of delay differential equations, 
as one can explicitly construct a Lyapunov functional
from  $U(t)$, where the solution is sometimes
referred to as delay Lyapunov matrices.
Sufficient conditions for  stability are given in
\cite{Kharitonov:2006:LYAPUNOV, Ochoa:2013:CRITICAL, Ochoa:2005:LYAPUNOV} and for neutral systems in
\cite{Ochoa:2007:NEUTRAL}, and
conditions for instability in 
\cite{Mondie:2011:INSTABILITY,Egorov:2014:NECESSARY}.
It has been used to provide bounds
on the transient phase of delay-differential equations in 
the PhD thesis \cite{Plischke:2005:TRANSIENT} and \cite{Kharitonov:2004:EXPEST,Kharitonov:2006:LYAPUNOVPLISCHKE}. Existence and uniqueness of the solutions are well characterized,
e.g., in \cite{Kharitonov:2006:LYAPUNOV}. See also the monograph \cite{Gu:2003:STABILITY}. Recently, it has been shown that in complete
analogy to the standard Lyapunov equation the solution
to the delay Lyapunov equation explicitly gives the $\mathcal{H}_2$-norm 
\cite{Jarlebring:2011:H2}. The delay Lyapunov equation 
can also be used to carry out a model order reduction 
which generalizes balanced truncation \cite{Jarlebring:2013:BALANCING}.

%The
%solutions are sometimes called delay Lyapunov matrices.
%The delay Lyapunov equations have a
% unique solution if the time-delay system
%\eqref{eq:tds} is asymptotically stable,
%which we will here assume. Literature
%on delay Lyapunov equations can
%be found in various generality settings in, e.g.,
%\cite{Kharitonov:2006:LYAPUNOVPLISCHKE,Kharitonov:2006:LYAPUNOV} [TODO:more-references].
%The explicit formula involving the matrix exponential was for the single-delay (and commensurate
%delays) given in \cite{Plischke:2005:TRANSIENT}. 
%The delay Lyapunov equations have been used for stability
%analysis (many references),  
%exponential bounds of time-delay systems
% \cite{Kharitonov:2004:EXPEST},
%$\mathcal{H}_2$-computation 
%\cite{Jarlebring:2011:H2} (+more recent results) and 
%model order reduction \cite{Jarlebring:2013:BALANCING}.
%

%Also many more results: \cite{Ochoa:2013:CRITICAL}
%\cite{Ochoa:2005:LYAPUNOV} 
%\cite{Ochoa:2007:NEUTRAL}
%\cite{Mondie:2011:INSTABILITY},
%\cite{Egorov:2014:NECESSARY} ... 
%
 
This paper concerns computational aspects of the
delay Lyapunov equation. 
Some computational aspects are treated in 
the literature, e.g., the matrix exponential formula 
in \cite{Plischke:2005:TRANSIENT},
the polynomial approximation 
approach in \cite{Huesca:2009:POLYNOMIAL},
spectral (Chebyshev-based) discretization approaches in 
\cite{Jarlebring:2011:H2,Vanbiervliet:2011:H2DISC}
and an ODE-approach in the PhD thesis
\cite[Chapter~3]{Merz:2012:THESIS}. 
%However, to our knowledge, there are no approaches
%that scale well with the size of the problem.

In complete contrast to the delay Lyapunov equation, the computational
aspects of the standard Lyapunov equation have received considerable
attention, mostly in the numerical linear algebra community.
Most importantly, the Bartels-Stewart method
\cite{Bartels:1972:LYAP}, 
ADI methods \cite{Benner:2007:ADI},
Krylov methods \cite{Simoncini:2007:KPIK,Hu:1992:KRYLOV},
and rational Krylov methods \cite{Jaimoukha:1994:KRYLOV}, 
including preconditioning techniques \cite{Hochbruck:1995:PRECONDITIONED}, have turned to
be effective in various situations. For a more thorough review, see the survey \cite{Simoncini:2015:LyapReview}.
%Although there are computational approaches for the delay Lyapunov equations,
%there is no analogue of the Bartel-Stewart algorithm \cite{Bartels:1972:LYAP}, ...  mathematical software such as Lyapack \cite{Penzl:2000:LYAPACK}
%Generalizations, e.g., Stykel \cite{Stykel:2010:KRYLOV}
%cite some more NLA-results  
To our knowledge, there exist no natural generalization of the Bartels-Stewart
algorithm and there are no Krylov methods for delay Lyapunov equation. 

The method we propose is tailored to medium-scale equations; it combines the use of a Krylov-type method and a direct algorithm similar to the Bartels-Stewart one.
More precisely, our approach is based on a characterization of the solution to 
the delay Lyapunov equation as a linear system of equations with $n^2$ unknowns. This characterization is derived in Section~\ref{sect:reformulation}.
Since the linear system derived in Section~\ref{sect:reformulation}
is large and only given implicitly as a matrix vector product, we
propose to adapt iterative methods which are based on
matrix vector products only, e.g., 
GMRES \cite{Saad:1986:GMRES} or 
BiCGStab \cite{Vorst:1992:BICGSTAB}, to this problem.
It turns out to be natural to use
a preconditioner involving a matrix equation called the 
T-Sylvester equation, for which there are efficient $O(n^3)$ methods for the dense case
\cite{DeTeran:2011:TSYLV}. We quantify the quality of the
preconditioner by deriving a bound on the convergence factor
of the iterative method. The iterative method and the preconditioner 
are given in Section~\ref{sect:algorithm}. 
The performance of the approach is illustrated
with simulations 
in Section~\ref{sect:simulations}. We apply the method
 to a problem stemming from the discretization of a two-dimensional
partial delay-differential equation (PDDE).  The 
number of iterations appears to be essentially 
independent of the grid, which suggests that the 
preconditioner is a sensible choice for this PDDE.

%* Description of our approach *

We use notation which is standard for analysis of matrix equations. 
The vectorization operation is denoted $\vec(B)$, i.e.,
if $B=\begin{bmatrix}b_1&\ldots&b_m\end{bmatrix}\in\RR^{n\times m}$,
$\vec(B)^T=\begin{bmatrix}b_1^T&\ldots&b_m^T\end{bmatrix}$.
The Kronecker product is denoted $\otimes$. Unless otherwise stated,
 $\|\cdot\|$ denotes the Euclidean norm for vectors and the
spectral norm for matrices. We denote the Frobenius norm 
by $\|\cdot\|_F$. 

%!TEX root = main_federico.tex
\begin{figure}
\centering
\includegraphics{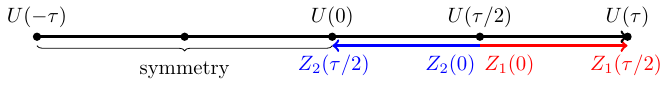}
\caption{Graphical representation of the relation between $U(t)$, $Z_1(t)$ and $Z_2(t)$.} \label{fig:line}
\end{figure}

\section{Reformulation of the delay Lyapunov equations}\label{sect:reformulation}
Our method is based on a reformulation of the
delay Lyapunov equation where we define for each $t\in[0,\tau/2]$
%\begin{subequations}
\begin{align}
  Z_1(t)&:=U(\tau/2+t), & 
  Z_2(t)&:=U(\tau/2-t).
\end{align}
%\end{subequations}
%
%
The two matrix-valued functions $Z_1(t)$ and $Z_2(t)$ coincide with $U(t)$ up to a change of the time coordinate which is represented visually in Figure~\ref{fig:line}. Essentially, they represent two different branches of $U(t)$ ``taking off'' from $\tau/2$ in opposite directions. Note that the left half of the function, $U([-\tau,0])$, is determined uniquely by the right half $U([0,\tau])$ by the transposition symmetry condition~\eqref{eq:dlyaps_sym}. The only nontrivial condition implied by~\eqref{eq:dlyaps_sym} is that $U(0)$ must be symmetric.

Note that 
\begin{subequations}
\begin{eqnarray}
Z_1(t-\tau)&=&U(t-\tau+\tau/2)=U(t-\tau/2)=
U(\tau/2-t)^T=Z_2(t)^T\\
Z_2(t-\tau)&=&U(\tau/2-t-\tau)=U(-t-\tau/2)=U(t+\tau/2)^T=Z_1^T(t)
\end{eqnarray}
\end{subequations}
Hence, the delay differential equation \eqref{eq:dlyapd_de} becomes
an ordinary differential equation
\begin{subequations}\label{eq:Z1Z2}
\begin{eqnarray}
  Z_1'(t)&=&\phantom{-}
%Z_1(t)A_0+ U(t+\tau/2-\tau)A_1=Z_1(t)A_0+ U(t-\tau/2)A_1=
Z_1(t)A_0+ Z_2(t)^TA_1,   \label{eq:Z1only} \\
  Z_2'(t)&=&-Z_1(t)^TA_1-Z_2(t)A_0. \label{eq:Z2only}
\end{eqnarray}
\end{subequations}
This is a constant-coefficient homogeneous linear system of ODEs which can be solved explicitly if the common (unknown) initial value $Z_1(0)=Z_2(0)=U(\tau/2)$ is provided. Using vectorization, we can give an explicit formula
\begin{equation} \label{eq:vectorizedODE}
\begin{bmatrix}
\vec Z_1(t)\phantom{^T}\\
\vec Z_2(t)^T\\
\end{bmatrix}=
\exp(t\mathcal{A})
\begin{bmatrix}
\vec U(\tau/2)\phantom{^T}\\
\vec U(\tau/2)^T
\end{bmatrix},
\end{equation}
where
\begin{equation}  \label{eq:defmathcalA}
\mathcal{A}:=
\begin{bmatrix}
  A_0^T\otimes I_n & A_1^T\otimes I_n\\
  -I_n\otimes A_1^T &   -I_n\otimes A_0^T 
\end{bmatrix}.
\end{equation}

In terms of $Z_1(t)$ and $Z_2(t)$, the algebraic condition
\eqref{eq:dlyaps_alg} and the symmetry condition 
\eqref{eq:dlyaps_sym} for $t=0$ reduce to
\begin{subequations} \label{conds}
\begin{eqnarray}
  0&=&W+
Z_2(\tau/2)^TA_0+A_0^TZ_2(\tau/2)+  \label{cond1}
Z_1(\tau/2)^TA_1+A_1^TZ_1(\tau/2),\\
  0&=&Z_2(\tau/2)-Z_2(\tau/2)^T.     \label{cond2}
\end{eqnarray}
\end{subequations}
Notice that the right-hand side of~\eqref{cond1} is symmetric and that of~\eqref{cond2} is antisymmetric. A linear combination of them gives
\begin{equation} \label{lincondZ}
0=W+Z_2(\tau/2)^T(A_0-c I)+(A_0^T+cI)Z_2(\tau/2)+
Z_1(\tau/2)^TA_1+A_1^TZ_1(\tau/2)	
\end{equation}
for each $c\in\mathbb{R}$, which forms the basis of our matrix operator.

%We note that another
%characterization of the delay 
%Lyapunov equations in \cite[Problem~6.38]{Plischke:2005:TRANSIENT}
%(and also Marshall) leads to the same ODE \eqref{eq:Z1Z2} by using the
%construction $\tilde{Z}_1(t)=U(t)$ and $\tilde{Z}_2(t)=U(t-\tau)$. The
%boundary conditions (initial condition) are different 
% and their formulation is not as useful for our purposes.
%

%In view of this formulation, we can define an operator as follows.
\begin{definition} Let $L_c:\RR^{n\times n}\rightarrow\RR^{n\times n}$ be defined by
  \begin{multline}  \label{eq:Ldef}
  L_c(X):=\\Z_2(\tau/2)^T(A_0-c I)+(A_0^T+cI)Z_2(\tau/2)+
Z_1(\tau/2)^TA_1+A_1^TZ_1(\tau/2)
  \end{multline}
where $Z_i:[0,\tau/2]\rightarrow\RR^{n\times n}$, $i=1,2$ are the unique solutions 
to the initial value problem \eqref{eq:Z1Z2}
with $Z_1(0)=Z_2(0)=X$.
\end{definition}

We shall need the following easy linear algebra result.
\begin{lemma} \label{symantisym}
Let $M=M^T\in\mathbb{R}^{n\times n}$ and $N=-N^T\in\mathbb{R}^{n\times n}$ be two matrices, one symmetric and one antisymmetric. Then, $M+N=0$ if and only if $M=N=0$.
\end{lemma}
\begin{proof} % temporarily commented out
The `if' part is trivial; let us prove the `only if'. Suppose $M+N=0$; then, transposing, we have also $0=M^T+N^T=M-N$. Summing and subtracting the two relations we have $2M=2N=0$.
\end{proof}

A time-delay system is called \emph{exponentially stable} if $\norm{x(t)} \leq \alpha\exp(-\beta t)$ for some constants $\alpha>0,\beta>0$. If this condition holds, then the solution $U(t)$ to~\eqref{eq:dlyaps} is unique \cite[Theorem~4]{Kharitonov:2006:LYAPUNOVPLISCHKE}. In this case, we can formulate the equivalence between the delay Lyapunov equation and a linear system with operator $L_c$.

\begin{theorem}[Equivalence] 
Suppose $A_0$ and $A_1$ and 
$\tau$ are such that \eqref{eq:tds} is
exponentially stable and let 
$W\in\RR^{n\times n}$ be any  symmetric matrix. 
Let $U$ be the solution to the 
delay Lyapunov equations \eqref{eq:dlyaps}
and let $L_c$ be defined by \eqref{eq:Ldef}.
Then, for any $c\neq 0$, $X=U(\tau/2)$ is the unique solution of the linear system 
\begin{equation}\label{eq:Lc_linsys}
L_c(X)=-W.
\end{equation}
\end{theorem}
\begin{proof}
Equation~\eqref{lincondZ} already shows that if $X=U(\tau/2)$ then $L_c(X)=-W$. It remains to prove the reverse implication. Suppose that $X$ satisfies~$L_c(X)+W=0$; then, by Lemma~\ref{symantisym} applied to
\begin{align*}
M &= Z_2(\tau/2)^TA_0+A_0^TZ_2(\tau/2)+
Z_1(\tau/2)^TA_1+A_1^TZ_1(\tau/2) - W,\\
N &= c(Z_2(\tau/2) - Z_2(\tau/2)^T),
\end{align*}
the conditions~\eqref{conds} hold. Define
\[
\hat{U}(t) = \begin{cases}
				Z_2(\tau/2-t) & 0 \leq t < \tau/2,\\
				Z_1(t-\tau/2) & \tau/2 \leq t \leq \tau,\\
				U(-t)^T &  -\tau\leq t <0.
			 \end{cases}
\]
The function $\hat{U}(t)$ is continuous in $0$ by~\eqref{cond2}, and in $\pm \tau/2$ by the choice of initial conditions, hence it is globally continuous on $[-\tau,\tau]$. Moreover, the differential equation~\eqref{eq:dlyapd_de} holds for all $t\neq 0, \tau/2$. By continuity, it must also hold for these values. Hence $\hat{U}(t)$ solves~\eqref{eq:dlyaps}. As we assume exponential stability, the solution is unique and hence $\hat{U}(t)=U(t)$.
\end{proof}

Since the linear system $L_c(X)=-W$ has a unique solution for each symmetric $W\in\mathbb{R}^{n\times n}$, we have
the following result.
\begin{corollary} \label{nonsing}
Suppose \eqref{eq:tds} is exponentially stable. Then, the linear operator $L_c$ is nonsingular for each $c\neq 0$.
\end{corollary}
A delay-free formulation of the delay Lyapunov
equations has also been derived in \cite[Equation (13)]{Kharitonov:2006:LYAPUNOV}. That formulation cannot be described with a linear operator 
in a way that can be adapted to an iterative method in the same
way that we show in the following section.

\section{Algorithm}\label{sect:algorithm}
%Now let $x=\vec(X)\in\RR^{n\times n}$ 
We now know from the previous section
that the matrix equation \eqref{eq:Lc_linsys} is
equivalent to the delay Lyapunov equation.
By vectorizing \eqref{eq:Lc_linsys}, we obtain the linear system 
on standard form
\begin{equation} \label{linsys}
\vec L_c(\unvec x)= -\vec W,	
\end{equation}
where the inverse function $\unvec(x)$ maps $\vec X \in\mathbb{R}^{n^2}$ to $X\in\mathbb{R}^{n\times n}$. Let $A\in\mathbb{R}^{n^2\times n^2}$ the matrix associated to it. We know that $A$ is nonsingular by Corollary~\ref{nonsing}. 
%Forming $A$ explicitly looks challenging, hence in this paper we wish to solve the linear system $Ax=-\vec W$ using an iterative method such as GMRES.

Our approach is based on specializing an iterative method
for linear systems to \eqref{linsys}.
In order to specialize  an iterative method
for large-scale linear systems, we need two ingredients.
We need an efficient procedure to compute the action corresponding
to the left-hand side of \eqref{linsys}; and we need
a preconditioner. These two ingredients are
described in the following two subsections.
\subsection{Action of $L_c$}\label{sect:action}
The action of the operator $L_c$ is defined by \eqref{eq:Z1Z2} and 
\eqref{eq:Ldef}. As a consequence,
the recipe to compute $L_c(X)$ for a given matrix $X$ is simple:
\begin{enumerate}
  \item Compute the solutions $Z_1(\tau/2)$, $Z_2(\tau/2)$ of the linear, constant-coefficient initial-value problem~\eqref{eq:Z1Z2} with initial values $Z_1(0)=Z_2(0)=X$. \label{propagationstep}
  \item Compute $L_c(X)$ using the expression~\eqref{eq:Ldef}.
\end{enumerate}
In practice, a detail is crucial in the choice of the numerical algorithm for the first step. We distinguish two possible scenarios:
\begin{itemize}
  \item We use a method with a fixed step-size and no adaptivity: for instance, the (explicit or implicit) Euler method, or a non-adaptive Runge-Kutta method. In this case, we are effectively substituting $L_c$ with a different operator $\hat{L}_c$, which replaces the differential operator in Step~\ref{propagationstep} with a finite discretization. This operator (for most classical methods) is still linear, so the theory of Krylov subspace methods can be applied without changes: we are applying a Krylov method to get an approximate solution of a nearby linear problem $\hat{L}_c$.
  \item We use an adaptive method, which can change step size along the algorithm, possibly in different ways for different initial values $X$. For instance, the Dormand-Prince method (Matlab's \texttt{ode45}). While apparently the two cases are similar, the addition of adaptivity has an important consequence: the computed operator $\hat{L}_c$, this time, is no longer a linear operator, because in general $\hat{L}_c(X_1+X_2) \neq \hat{L}_c(X_1) + \hat{L}_c(X_2)$. Indeed, for different values of the input $X$ the initial-value problems could be solved using different grids, and hence different discrete approximations of the propagation operator. The correct framework to analyze the method in this case is the one of inexact Krylov methods~\cite{Simoncini:2003:INEXACTKRYLOV}. We present an error analysis under this framework in Section~\ref{sect:inexact}.
\end{itemize}

\subsection{Preconditioning}\label{sect:preconditioning}
In order to make iterative methods effective, 
it is common to carry out a transformation which 
preconditions the problem. This can often be interpreted as 
transforming the problem with an approximation of the
 inverse of the matrix/operator.
We focus on a particular preconditioner 
obtained by solving the problem exactly when $A_1$ is replaced with the zero matrix. Then~\eqref{eq:Ldef} becomes
\begin{equation}\label{eq:tLdef}
\tilde{L}_c(X) := Z_2(\tau/2)^T(A_0-c I)+(A_0^T+cI)Z_2(\tau/2),
\end{equation}
and~\eqref{eq:Z2only} decouples from $Z_1$ such that 
\begin{equation}\label{eq:Ztilde_approx}
Z_2' = -Z_2(t)A_0,
\end{equation}
which we can solve explicitly to get $Z_2(\tau/2) = X\exp(-\tau A_0/2)$.

Let $T$ be the operator 
\[
  T(Y)=(A_0^T+cI)Y+Y^T(A_0-cI).
\]
The operator $L_c$ is invertible if and only $T^{-1}$ exists, and in this case we have
\begin{equation} \label{eq:defLtildeInv}
 \tilde{L}^{-1}_c(Z)=T^{-1}(Z)\exp(\tau A_0/2). 
\end{equation}
Inverting the operator $T$ correspond to solving the so-called (real) \emph{T-Sylvester equation} $MY+Y^TN=C$. The paper~\cite{DeTeran:2011:TSYLV} discusses the solvability of this equation and presents a direct $O(n^3)$ Bartels--Stewart-like algorithm for its solution. In particular, the following result holds.
\begin{theorem}[\protect{\cite[Lemma~8]{Kressner:2009:PALQR},\cite{DeTeran:2011:TSYLV}}]
Let $M,N,C\in\mathbb{R}^{n\times n}$. The equation $MX+X^TN=C$ has a unique solution $X$ for each right-hand side $C$ if and only if $\mu_i\bar{\mu}_j\neq 1$ for each pair $\mu_i,\mu_j$ of eigenvalues of the pencil $M-\lambda N^T$.
\end{theorem}

In our case, $M=A_0^T+cI$, $N=A_0-cI$, so after a quick computation the solvability condition reduces to the following condition, which is independent of $c$.

\begin{definition}[Hamiltonian eigenpairing]
We say that the matrix $A_0\in\RR^{n\times n}$ has no \emph{Hamiltonian eigenpairing}, if for each pair of eigenvalues $\lambda_i,\lambda_j$ of the matrix $A_0$, we have
\[
\lambda_i+\bar{\lambda}_j\neq 0.
\]
\end{definition}
%\begin{itemize}
%  \item \textbf{Condition~S}: For each pair of eigenvalues $\lambda_i,\lambda_j$ of the matrix $A_0$, it holds that $\lambda_i+\bar{\lambda}_j\neq 0$.
%\end{itemize}
A matrix has no Hamiltonian eigenpairing, for instance, if $\Re \lambda < 0$ for each eigenvalue $\lambda$ of $A_0$, i.e., if the delay-free system obtained by setting $A_1=0$ is stable.

%We call $\norm{M}_2$ and $\norm{M}_F$ the operator (Euclidean) and Frobenius norm of a matrix $M$, respectively. % already introduced in intro

In order to characterize the convergence 
and quality of the preconditioner we use
a fundamental min-max bound. 
Suppose we carry out GMRES on the matrix $A\in\RR^{N\times N}$ 
with eigenvalues $\lambda_1,\ldots,\lambda_N$. 
From  \cite[Proposition~4]{Saad:1986:GMRES} we have
the bound of the residual
\[
  \|r_{m+1}\|\le \kappa(V)\varepsilon^{(m)}\|r_0\|,
\]
where $V$ is the eigenvector matrix of $A$ (which is assumed to be diagonalizable),
and
\[
  \varepsilon^{(m)}=\min_{p\in P_m} \max_i |p(\lambda_i)|
\]
where $P_m=\{p:\textrm{polynomial of degree }m\textrm{ such that }p(0)=1\}$. 
We now apply the standard Zarantonello
 bound~\cite[Lemma~6.26]{Saad:1996:LINSYS}, where we assume that the
eigenvalues are contained in a disk of radius $r$ centered
at $c=1$, corresponding to selecting $p(z)=\frac{(c-z)^m}{c^m}$
such that $\varepsilon^{(m)}\le r^m/c^m=r^m\le \|A-I\|^m$. 
Preconditioned GMRES with preconditioner $\tilde{A}^{-1}$
is equivalent to GMRES in exact arithmetic applied to the matrix $\tilde{A}^{-1}A$ (apart from
termination criteria and initialization). Therefore,  a bound on $\|\tilde{A}^{-1}A-I\|$ provides
a characterization of the  convergence factor of preconditioned GMRES. 
Because of the vectorization included in our setting, bounding $\|\tilde{A}^{-1}A-I\|$ corresponds to giving an estimate for the quantity
\[
  \frac{\|\tilde{L}_c^{-1}(L_c(X))-X\|_F}{\norm{X}_F}.
\]
Our preconditioner is constructed by setting $A_1=0$.
Therefore, we expect that 
the preconditioner works well if $\|A_1\|$ is small.
This reasoning is formalized in the following result.
%A sufficient condition for a preconditioner
%is that the composition of the original operator
%and the preconditioner is close to unity, 
%which is formalized in  the following result. 

%\[
%  
%\]
%is the set of polynomials of degree $m$ normalized
%such that $p(0)=1$ when $p\in P_m$. 
%* reference to bound on convergence factor in terms
%of norm of $MA-I$, where $M$ is the precondition * 
\begin{theorem}[Quality of preconditioner]\label{thm:precondquality}
 Suppose the system \eqref{eq:tds} is exponentially stable and suppose that $A_0$ 
has no Hamiltonian eigenpairing. Let
$L_c$ and $\tilde{L}_c$ be defined by \eqref{eq:Ldef} and \eqref{eq:tLdef} respectively. Then,
\begin{equation}\label{eq:precondquality}
  \frac{\|\tilde{L}_c^{-1}(L_c(X))-X\|_F}{\norm{X}_F}=\mathcal{O}(\|A_1\|_2),
\end{equation}
where the constant hidden in the $\mathcal{O}(\cdot)$ notation depends only on $\norm{A_0}$, $\tau$ and $c$.
\end{theorem}
\begin{proof}
We invoke Lemma~\ref{thm:tLc_bound} (provided in~\ref{sec:appendixtechnical})
to bound the left-hand side of \eqref{eq:precondquality}
\begin{multline}\label{eq:firstbound}
\frac{\norm*{\tilde{L}_c^{-1}\left(L_c(X)\right)-X}_F}{\norm{X}_F}  = 
\frac{\norm*{\tilde{L}_c^{-1}\left(L_c(X) - \tilde{L}_c(X)\right)}_F}{\norm{X}_F} \leq\\
K\exp(\tau\|A_0\|/2)
\frac{\norm*{L_c(X) - \tilde{L}_c(X)}_F}{\norm{X}_F}.
% K \exp(\tau\|A_0\|_F/2) \mathcal{O}(\|A_1\|_2) = \mathcal{O}(\|A_1\|_2).
\end{multline}

In order to bound $L_c(X) - \tilde{L}_c(X)$
we let $Z_1$ and $Z_2$ correspond to $L_c(X)$, i.e., 
they satisfy the equations \eqref{eq:Z1Z2} with initial value $Z_1(0)=Z_2(0)=X$.
We use tilde for the differential equation corresponding to  $\tilde{L}_c(X)$, i.e.,
$\tilde{Z}_2(t)$ satisfies \eqref{eq:Ztilde_approx}.
%Let $\tilde{Z}_1$ and $\tilde{Z}_2$ be the corresponding quantities for $\tilde{L}_c(X)$.
%Hence, for $i=1$ and $i=2$, 
%\[
%  \|\tilde{Z}_i(t)\|\le C \exp(\tau \|A_0\|)\|X\|.
%\]
Moreover, let $\Delta_2:=Z_2-\tilde{Z}_2$. We have
\begin{multline}\label{eq:tLcLc}
  \tilde{L}_c(X)-L_c(X)=\\
\Delta_2(\tau/2)^T(A_0-c I)+(A_0^T+cI)\Delta_2(\tau/2)+
Z_1(\tau/2)^TA_1+A_1^T Z_1(\tau/2),
\end{multline}
for which $\Delta_2(\tau/2)$ and $Z_1(\tau/2)$
 can be bounded as follows.
Lemma~\ref{lem:Zbound} provided in~\ref{sec:appendixtechnical} tells us that
\begin{equation} \label{eq:Zbound}
  \|Z_1(\tau/2)\|_F \le 2\exp(\tau(\|A_0\|_2+\|A_1\|_2))\|X\|_F.
\end{equation}
By definition, $\Delta_2$ satisfies the ODE
\begin{equation}\label{eq:Delta_ODE}
 \Delta_2'(t) =-\Delta_2(t)A_0 +g(t), \quad
 \Delta_2(0)=0,
 \end{equation}
where $g(t):=-Z_1(t)^TA_1.$
The  variation-of-constants formula applied to \eqref{eq:Delta_ODE} results in the explicit expression
\[
 \Delta_2(t)=
-\int_0^t Z_1(s)^TA_1\exp((s-t)A_0)\,ds.
\]
Hence,
\begin{subequations} \label{eq:Delta_bound}
\begin{align}
\|\Delta_2(\tau/2)\|_F &\le \int_0^{\tau/2} \norm{Z_1(s)^TA_1\exp((s-\tau/2)A_0)}_F\,ds\\
 & \le \int_0^{\tau/2} \norm{Z_1(s)}_F \norm{A_1}_2\norm{\exp((s-\tau/2)A_0)}_2\,ds\\
 & \le \tau \exp(\tau(\|A_0\|_2+\|A_1\|_2))\norm{A_1}_2\exp(\tau\norm{A_0}_2/2) \|X\|_F.
\end{align}
\end{subequations}
We now evaluate the Frobenius norm of \eqref{eq:tLcLc} 
and apply the triangle inequality and the bounds~\eqref{eq:Zbound} and~\eqref{eq:Delta_bound},
which shows that
\begin{equation}\label{eq:tLcLc_diff}
  \frac{\|\tilde{L}_c(X)-L_c(X)\|_F}{\norm{X}_F}=\mathcal{O}(\|A_1\|_2).
\end{equation}
The hidden constant in \eqref{eq:tLcLc_diff} depends only on $\norm{A_0}_2$, $c$, and $\tau$. The conclusion \eqref{eq:precondquality}
follows by combining \eqref{eq:firstbound} and \eqref{eq:tLcLc_diff}.
%Finally, using Lemma~\ref{thm:tLc_bound}, we get
%\begin{multline}
%\frac{\norm*{\tilde{L}_c^{-1}\left(L_c(X)\right)-X}_F}{\norm{X}_F}  = 
%\frac{\norm*{\tilde{L}_c^{-1}\left(L_c(X) - \tilde{L}_c(X)\right)}_F}{\norm{X}_F} \leq\\ K \exp(\tau\|A_0\|_F/2) \mathcal{O}(\|A_1\|_2) = \mathcal{O}(\|A_1\|_2).
%\end{multline}
\end{proof}

\subsection{Inexact Krylov theory} \label{sect:inexact}

As described in Section~\ref{sect:action}, if one uses an adaptive method for the integration, then assessing convergence requires the theory of inexact Krylov methods. The \emph{inexact GMRES} method for an operator $A$ is defined as the classical GMRES iteration, but with the difference that at each step $i = 1,2,\dots,k$ we do not compute the action of $w_i = A v_i$ of $A$ on a vector $v_i$, but rather we replace it with an approximation $w^{\mathrm{inex}}_i = (A + E_i) v_i$, for an unknown matrix $E_i$. The matrix $E_i$ can vary at each iteration. In equivalent terms, we can say that the product $A v_i$ is computed up to a specified accuracy $\norm{E_i}$, since
\[
\frac{\norm{w^{\mathrm{inex}}_i - A v_i}}{\norm{v_i}} = \frac{\norm{E_iv_i}}{\norm{v_i}} \leq \norm{E_i}.
\]
This process produces a Hessenberg matrix $H^{\mathrm{inex}}_i$, a sequence of approximations $x^{\mathrm{inex}}_i$ to the solution of the linear system, and a sequence of `fake' residuals $r^{\mathrm{inex}}_i$; these fake residual values are the ones computed during the iterative method, and they do \emph{not} equal in general $b - A x^{\mathrm{inex}}_i$. However, the following result holds.
\begin{theorem}[\protect{\cite[Theorem~5.3]{Simoncini:2003:INEXACTKRYLOV}}] \label{inexacttheorem}
Assume that $k \leq m$ iterations of the inexact GMRES method on an operator $A\in\mathbb{C}^{m\times m}$ have been carried out, and that for some $\delta > 0$ we have
\[
\norm{E_i} \leq \frac{\sigma_{\min}(H^{\mathrm{inex}}_k)}{k}\frac{1}{\norm{r^{\mathrm{inex}}_{i-1}}}\delta, \quad i = 1,2,\dots,k.
\]
Then, $\norm{b - A x_k^{\mathrm{inex}} - r^{\mathrm{inex}}_{k}} \leq \delta$. 
\end{theorem}
We would like to use this result to apply an ODE solver to compute an approximation $\hat{L}_c$ to the operator $L_c$, and tuning its accuracy at each step. However, this result is somehow ineffective for a truly adaptive computation: given a target error $\delta$, the accuracy at which we need to perform the matrix-vector product at step $i$ in order to obtain it is not available until the final step. Instead, we proceed as follows. Given a target accuracy goal $\varepsilon$, we apply several steps of the inexact GMRES method, and at each step $i=1,2,\dots$ we tune its accuracy so that
\[
\norm{E}_i \leq \frac{C\varepsilon}{\norm{r^{\mathrm{inex}}_{i-1}}},
\]
for a given constant $C$, and we stop the method at the first step $k$ for which $\norm{r^{\mathrm{inex}}_{i-1}} \leq \varepsilon$. Applying Theorem~\ref{inexacttheorem} with $\delta = \frac{k}{\sigma_{\min}(H^{\mathrm{inex}}_k)} C\varepsilon$ and the triangle inequality we obtain
\[
\norm{b - A x_k^{\mathrm{inex}}} \leq \norm{r^{\mathrm{inex}}_k} + \frac{k}{\sigma_{\min}(H^{\mathrm{inex}}_k)} C\varepsilon.
\]

The problem of computing the preconditioned operator $\tilde{L}_c^{-1}L_c$ up to a given accuracy is in itself nontrivial. Algorithms for adaptive integration of initial-value problems such as Matlab's \texttt{ode45} can produce $(\tilde{Z}_1(\tau/2),\tilde{Z}_2(\tau/2))$ such that
\[
\norm*{
\begin{bmatrix}
  Z_1(\tau/2) - \tilde{Z}_1(\tau/2)\\
  Z_2(\tau/2) - \tilde{Z}_2(\tau/2)
\end{bmatrix}
}_F \leq \varepsilon \norm*{\begin{bmatrix}
  Z_1(\tau/2)\\
  Z_2(\tau/2)
\end{bmatrix}}_F
\]
for a given threshold $\varepsilon$; however, even before taking into account the preconditioner, computing 
\[
\tilde{Z}_2(\tau/2)^T(A_0-c I)+(A_0^T+cI)\tilde{Z}_2(\tau/2)+
\tilde{Z}_1(\tau/2)^TA_1+A_1^T\tilde{Z}_1(\tau/2)
\]
may amplify this error by a coefficient which is difficult to bound \emph{a priori}. Hence we can only obtain a very weak result: \emph{if} integrating the ODE~\eqref{eq:Z1Z2} with relative accuracy $\frac{\varepsilon}{\norm{r^{\mathrm{inex}}_{i-1}}}$ produces a relative error in $\tilde{L}_c(L_c(X))$ which is bounded by $\frac{C\varepsilon}{\norm{r^{\mathrm{inex}}_{i-1}}}$ for some constant $C$, then the residual of the computed solution satisfies
\[
\norm{L_c(X) + W}_F \leq \norm{r^{\mathrm{inex}}_k} + \frac{k}{\sigma_{\min}(H^{\mathrm{inex}}_k)} C\varepsilon.
\]

\subsection{A residual measure} \label{sec:residual}
It is useful to have a method to assess the accuracy of a computed solution to the system~\eqref{eq:dlyaps}. This is a nontrivial task: first of all, this is a system of delay differential equations, so trying to evaluate it on a computer requires careful approximation; moreover, even ignoring this fact, due to the nontrivial coupling conditions between the values of the function in the two parts of the interval $[0,\tau]$, it is not immediate to choose a $n\times n$ initial value, integrate the equations, and produce an associated $W$ which we can use to test the methods on a problem for which we know the exact solution.

To this purpose, we suggest a residual measure as follows. Given approximations $\tilde{U}_0 \approx U(0),\tilde{U}_\tau \approx U(\tau)$ computed by a numerical method, we check that:
\begin{itemize}
	\item integrating numerically with \texttt{ode45} the ODE~\eqref{eq:Z1Z2} from the initial value $t=\tau/2$ $Z_1(\tau/2) = \tilde{U}_\tau, Z_2(\tau/2) = \tilde{U}_0$ to $t=0$ produces values $Z_1(0),Z_2(0)$ such that $r_1 := \norm{Z_1(0)-Z_2(0)}_F$ is small (compared to $s_1:=\norm{Z_1(0)}_F$);
	\item $\tilde{U}_0$ is such that $r_2 := \norm{\tilde{U}_0-\tilde{U}_0^\TT}_F$ is small (compared to $s_2 := \norm{\tilde{U}_0}_F$); and
	\item the quantity $r_3 := \norm{\tilde{U}_0A_0+A_0^\TT \tilde{U}_0+ \tilde{U}_\tau^\TT A_1+A_1^\TT \tilde{U}_\tau + W}_F$ is small (compared to $s_3 := \norm{W}_F$).
\end{itemize}
We use the Frobenius norm here since we care about speed of computation when $n$ may reach the order of thousands. To avoid issues in cases where one of the $s_i$ is very small and hence its relative residual may be large, we define a global residual measure as
\[
\operatorname{res}(\tilde{U}_0,\tilde{U}_\tau) := \frac{r_1+r_2+r_3}{s_1+s_2+s_3}.
\]
This residual measure is built on approximations to $U(0)$ and $U(\tau)$ as its inputs. It is indeed possible to construct an analogous measure starting from an approximation to $U(\tau/2)$ instead, which
may look more natural in view of the development in the previous sections. However, a reader looking with critical eye may wonder if the good results obtained by the methods introduced here are due to the choice of a residual function that favors the midpoint $U(\tau/2)$ over the endpoints $U(0)$ and $U(\tau)$, since our method builds heavily on $U(\tau/2)$, while it is not a quantity that appears naturally in the competing algorithms. Thus we choose to work with $\tilde{U}_0, \tilde{U}_\tau$ to get a fairer assessment of the merits of this method.

\section{Simulations}\label{sect:simulations}
\subsection{A small example}\label{sect:smallexample}

In order to illustrate the preconditioner and properties of our approach
we first consider a small example with randomly generated $A_0$ matrix.
We specify the matrices for reproducibility 
\[ 
A_0=
\begin{bmatrix}
   -26 &   22  &  -1 &   -4\\
     2 &  -24  &  -4 &    1\\
     7 &   11  & -24 &  -22\\
   -13 &   15  &  -1 &   -9
\end{bmatrix},\;\;
A_1=\alpha\diag(-1,-0.5,0,0.5),\;\;W=I
\] 
and $\tau=1$. We carry out simulations for different $\alpha=\|A_1\|$.
The time-delay system is stable for all $\alpha\in[0,10]$. 
The corresponding delay Lyapunov equation satisfies
\[
 U(\tau/2)\approx
 \frac{1}{100}\cdot\begin{bmatrix}
    0.2302& -0.0156&   0.0101 &  -0.3729\\
   -0.0885&  0.0044&  -0.0038 &   0.1380\\
    0.1466& -0.0057&   0.0056 &  -0.2263\\
   -0.5485&  0.0331&  -0.0238 &   0.8755
 \end{bmatrix}
\]
for $\alpha=1$.
%Note that both Krylov methods 
%will terminate in at most $n^2=16$ iterations in exact arithmetic.

We combine our approach with two different generic iterative methods for linear systems
of equations, GMRES \cite{Saad:1986:GMRES} and BiCGStab \cite{Vorst:1992:BICGSTAB} and select $c=1$.
To illustrate the properties of the performance
of the iterative method, we solve the ODE defining
$L_c$ to full precision with the matrix exponential.
The absolute error as a function of iteration is given in 
Figure~\ref{fig:experiment_small_convergence}.
 Both methods successfully solve the
problem before the break-down at iteration
$n^2$ except for $\|A_1\|=10$.
No substantial difference between the two iterative
methods can be observed
in the error as a function of iteration, i.e., 
 nothing can be concluded regarding which of
the two variants is better
for this problem. The convergence  of
the two methods is faster for small $\|A_1\|$. This
is due to the fact that the preconditioner
is more effective when $\|A_1\|$ is small,
which is consistent with Theorem~\ref{thm:precondquality} 
and Figure~\ref{fig:precondquality}, where
we clearly see that the norm of the
 preconditioned
system $X\mapsto\tilde{L}_c^{-1}(L_c(X))$ 
has a linear dependence on $\|A_1\|$. The same
conclusion is supported by the localization of the
eigenvalues of the linear map $X\mapsto\tilde{L}_c^{-1}(L_c(X))$
in Figure~\ref{fig:precondquality}b.

%All of the following for $\tau=10$, $\tau=5$, $\tau=1$: 
%* illustration with GMRES, bicgstab with high precision solution
%to ODE.

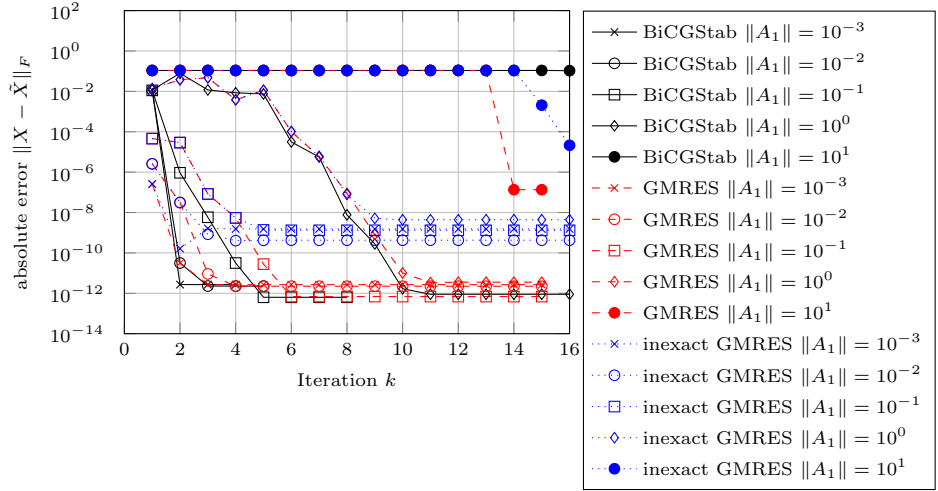
\begin{figure}[h]
  \begin{center}
\tikzsetnextfilename{experiment_small_convergence} % name for the temporary generated file by tikz externalization in gfx_tmp (see preamble)
\begin{tikzpicture}
\begin{semilogyaxis}%
[width=7.5cm,% 6cm normally enough space to use as two subfigures
 height=5.88cm,% 6cm normally enough space to use as two subfigures 
 xlabel={Iteration $k$},%
 ylabel={absolute error $\|X-\tilde{X}\|_F$},%
 % y label style={yshift=-3pt},% y label too close to axis?
 % xtick=data, ytick=data, % if you want data to determine ticks
 xmin=0,xmax=16,%
 ymin=1e-14,ymax=1e2,%
 y tick label style={text width=\widthof{$10^{-10}$}}, %align the y-ticks to the ``left''
 xtick={0,2,4,...,16},%
 ytick={1e-14,1e-12,1e-10,1e-8,1e-6,1e-4,1e-2,1,1e2},
 legend entries={
BiCGStab $\|A_1\|=10^{-3}$, 
BiCGStab $\|A_1\|=10^{-2}$, 
BiCGStab $\|A_1\|=10^{-1}$, 
BiCGStab $\|A_1\|=10^0$, 
BiCGStab $\|A_1\|=10^1$, 
GMRES $\|A_1\|=10^{-3}$, 
GMRES $\|A_1\|=10^{-2}$, 
GMRES $\|A_1\|=10^{-1}$, 
GMRES $\|A_1\|=10^0$, 
GMRES $\|A_1\|=10^1$, 
inexact GMRES $\|A_1\|=10^{-3}$, 
inexact GMRES $\|A_1\|=10^{-2}$, 
inexact GMRES $\|A_1\|=10^{-1}$, 
inexact GMRES $\|A_1\|=10^0$, 
inexact GMRES $\|A_1\|=10^1$, 
},%
 %legend pos=north west, legend style={draw,cells={anchor=west},row sep=-3pt},%
 legend pos=outer north east, %legend style={draw=none,fill=white,cells={anchor=west},row sep=-2pt,font=0.1},%
 %mark repeat={3},% skip drawing the mark every X timme
 legend cell align=left,
 grid,
 domain=0:100,
 %clip mode=individual,% ?
]
%\addplot[blue,solid] table[x index=0,y index=1]{gfx/convergence_semilogy5.txt};
%\addplot[blue,dashed] table[x index=0,y index=1]{gfx/convergence_semilogy4.txt};
%\addplot[blue,dashdotted] table[x index=0,y index=1]{gfx/convergence_semilogy3.txt};
%% BICGSTAB
\addplot[black, solid,mark=x, mark options={solid,scale=1}] table[x index=0,y index=1]{gfx/experiment_small_convergence15.txt};
\addplot[black, solid,mark=o, mark options={solid,scale=1}] table[x index=0,y index=1]{gfx/experiment_small_convergence12.txt};
\addplot[black, solid,mark=square, mark options={solid,scale=1}] table[x index=0,y index=1]{gfx/experiment_small_convergence9.txt};
\addplot[black, solid,mark=diamond, mark options={solid,scale=1}] table[x index=0,y index=1]{gfx/experiment_small_convergence6.txt};
\addplot[black, solid,mark=*, mark options={solid,scale=1}] table[x index=0,y index=1]{gfx/experiment_small_convergence3.txt};
%% GMRES
\addplot[red, dashed,mark=x, mark options={solid,scale=1}] table[x index=0,y index=1]{gfx/experiment_small_convergence14.txt};
\addplot[red, dashed,mark=o, mark options={solid,scale=1}] table[x index=0,y index=1]{gfx/experiment_small_convergence11.txt};
\addplot[red, dashed,mark=square, mark options={solid,scale=1}] table[x index=0,y index=1]{gfx/experiment_small_convergence8.txt};
\addplot[red, dashed,mark=diamond, mark options={solid,scale=1}] table[x index=0,y index=1]{gfx/experiment_small_convergence5.txt};
\addplot[red, dashed,mark=*, mark options={solid,scale=1}] table[x index=0,y index=1]{gfx/experiment_small_convergence2.txt};
%% inexact GMRES
\addplot[blue, dotted,mark=x, mark options={solid,scale=1}] table[x index=0,y index=1]{gfx/experiment_small_convergence13.txt};
\addplot[blue, dotted,mark=o, mark options={solid,scale=1}] table[x index=0,y index=1]{gfx/experiment_small_convergence10.txt};
\addplot[blue, dotted,mark=square, mark options={solid,scale=1}] table[x index=0,y index=1]{gfx/experiment_small_convergence7.txt};
\addplot[blue, dotted,mark=diamond, mark options={solid,scale=1}] table[x index=0,y index=1]{gfx/experiment_small_convergence4.txt};
\addplot[blue, dotted,mark=*, mark options={solid,scale=1}] table[x index=0,y index=1]{gfx/experiment_small_convergence1.txt};

\end{semilogyaxis}
\end{tikzpicture}
    \caption{
      Convergence for different preconditioned iterative methods
applied to the small example in Section~\ref{sect:smallexample}. The tolerance for the inexact solver is $\varepsilon=10^{-10}$.
      \label{fig:experiment_small_convergence}
    }
  \end{center}
\end{figure}

\begin{figure}[h]
  \begin{center}
\subfigure[Difference in norm between the preconditioned matrix and the identity]{
\tikzsetnextfilename{experiment_small_norms} % name for the temporary generated file by tikz externalization in gfx_tmp (see preamble)
\begin{tikzpicture}
\begin{loglogaxis}%
[width=6cm,% 6cm normally enough space to use as two subfigures
 height=5.88cm,% 6cm normally enough space to use as two subfigures 
 ylabel={$\underset{X\in\CC^{n\times n}}{\max}\frac{\|\tilde{L}_c^{-1}(L_c(X))-X\|_F}{\|X\|_F}$},%
 xlabel={$\|A_1\|$},%
 % y label style={yshift=-3pt},% y label too close to axis?
 % xtick=data, ytick=data, % if you want data to determine ticks
 ymin=1e-6,ymax=1e6,%
 xmin=1e-5,xmax=1e2,%
 y tick label style={text width=\widthof{$10^{-}$}}, %align the y-ticks to the ``left''
 xtick={1e-4,1e-2,1e0,1e2},%
 ytick={1e-6,1e-4,1e-2,1e0,1e2,1e4,1e6},
% legend entries={},%
 %legend pos=north west, legend style={draw,cells={anchor=west},row sep=-3pt},%
 legend pos=outer north east, %legend style={draw=none,fill=white,cells={anchor=west},row sep=-2pt,font=0.1},%
 %mark repeat={3},% skip drawing the mark every X timme
 legend cell align=left,
 grid,
 domain=0:100,
 %clip mode=individual,% ?
]
\addplot[black, solid,mark=none, mark options={solid,scale=1}] table[x index=0,y index=1]{gfx/precond_visualization2.txt};
\end{loglogaxis}
\end{tikzpicture}%
}%
\subfigure[Difference in eigenvalue location]{
	\tikzsetnextfilename{experiment_small_norms} % name for the temporary generated file by tikz externalization in gfx_tmp (see preamble)
\begin{tikzpicture}
\begin{loglogaxis}%
[width=6cm,% 6cm normally enough space to use as two subfigures
 height=5.88cm,% 6cm normally enough space to use as two subfigures 
 ylabel={$\max \abs{\lambda-1}$},%
 xlabel={$\|A_1\|$},%
 % y label style={yshift=-3pt},% y label too close to axis?
 % xtick=data, ytick=data, % if you want data to determine ticks
 ymin=1e-6,ymax=1e6,%
 xmin=1e-5,xmax=1e2,%
 y tick label style={text width=\widthof{$10^{-}$}}, %align the y-ticks to the ``left''
 xtick={1e-4,1e-2,1e0,1e2},%
 ytick={1e-6,1e-4,1e-2,1e0,1e2,1e4,1e6},
% legend entries={},%
 %legend pos=north west, legend style={draw,cells={anchor=west},row sep=-3pt},%
 legend pos=outer north east, %legend style={draw=none,fill=white,cells={anchor=west},row sep=-2pt,font=0.1},%
 %mark repeat={3},% skip drawing the mark every X timme
 legend cell align=left,
 grid,
 domain=0:100,
 %clip mode=individual,% ?
]
\addplot[black, solid,mark=none, mark options={solid,scale=1}] table[x index=0,y index=1]{gfx/precond_visualization1.txt};
\end{loglogaxis}
\end{tikzpicture}%	
}
    \caption{Illustration of the quality of 
the preconditioner.
      \label{fig:precondquality}
    }
  \end{center}
\end{figure}
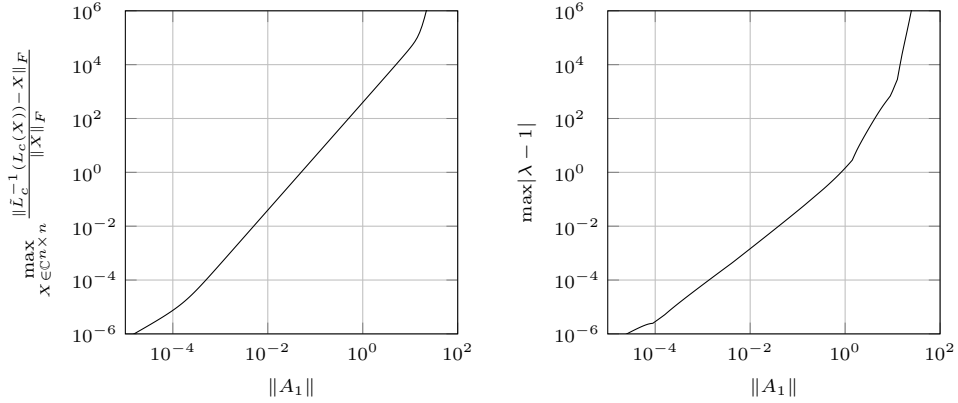

\subsection{A large-scale example}\label{sect:largeexample}
In relation to other methods
for delay Lyapunov equations, our iterative approach
is likely to have better relative performance
for large problems. We illustrate this with 
the following  time-delay
system stemming from the discretization
of a partial differential equation with
delay\footnote{The Matlab code for the example and
the simulation is publicly available on \url{http://www.math.kth.se/~eliasj/src/dlyap_precond}}. More precisely, we consider
on the domain  $(x,y)\in[0,1]\times[0,1]$ the
PDDE
\begin{subequations}\label{eq:pdde}
\begin{eqnarray}
  \ddot{v}(x,y,t)&=&\Delta v(x,y,t)
+\dot{v}(x,y,t)+
f(x,z)\frac{\partial v}{\partial x}(x,y,t-\tau)+u(t)\\
  w(t)&=&v(1/2,1/2)
\end{eqnarray}
\end{subequations}
where $f(x,y)=f_0\cos(xy)\sin(\pi x)$ 
with homogeneous Dirichlet boundary conditions, and $f_0=5$.
The PDDE \eqref{eq:pdde} can be interpreted as 
waves propagating on a square, with damping and 
delayed feedback control.
PDDEs are for instance studied in  \cite{Wu:1996:PFDE}.
In order to reach a problem of the form \eqref{eq:tds}
we rephrase \eqref{eq:pdde} as a 
 system of PDDEs which is first-order in time.
We carry out a semi-discretization with finite differences
in space 
with $n_x+1$ intervals in the $x$-direction 
and $n_y+1$ intervals in the $y$-direction,
i.e., 
$h_x=1/(n_x+1)$, $x_{k}=kh_x$, $k=1,\ldots,n_x$
and 
$h_y=1/(n_y+1)$, $y_{k}=kh_y$, $k=1,\ldots,n_y$.
The corresponding discretized time-delay system
is of the form 
\eqref{eq:tds} with coefficient matrices given by
\begin{subequations}
\begin{eqnarray}
  A_0&=&
  \begin{bmatrix}
   0 & I\\
   I\otimes D_{xx}+ D_{yy}\otimes I & -I
  \end{bmatrix}
   \\
  A_1&=& 
  \begin{bmatrix}
    0 & 0  \\
    \diag(F)(I\otimes D_x) & 0
  \end{bmatrix}
\\
  B_0 &=&
\begin{bmatrix}1&\cdots&1&0&\cdots&0\end{bmatrix}^T\\
%\begin{bmatrix}1&\cdots&1\end{bmatrix}^T\\
  C_0 &=&
  \begin{bmatrix}
e_{(n_y+1)/2}^T\otimes e_{(n_x+1)/2}^T &0&\cdots&0\\
  \end{bmatrix}
\end{eqnarray}
\end{subequations}
where 
%and $D_{xx}$, $D_{yy}$ correspond
%to the discretization
%of the second-order finite-difference
%approximation
\begin{eqnarray*}
 D_{xx}&=&\frac{1}{h_{x}^2}
 \begin{smallbmatrix}
   -2& 1      &      & \\
    1& \ddots &\ddots& \\
     & \ddots &\ddots&1\\
     &        & 1    & -2
 \end{smallbmatrix}\in\RR^{n_x\times n_x},\;
 D_{yy}=
\frac{1}{h_y^2}
 \begin{smallbmatrix}
   -2& 1      &      & \\
    1& \ddots &\ddots& \\
     & \ddots &\ddots&1\\
     &        & 1    & -2
 \end{smallbmatrix}\in\RR^{n_y\times n_y},\\
 D_{x}&=&\frac{1}{2h_x}
 \begin{smallbmatrix}
    0 & 1      &      & \\
    -1& \ddots &\ddots& \\
     & \ddots & \ddots&1\\
     &        & -1    & 0
 \end{smallbmatrix}\in\RR^{n_x\times n_x},\;\; F=\vec([f(x_i,y_j)]_{i,j=1}^{n_x,n_y}).
\end{eqnarray*}
In the setting of $\mathcal{H}_2$-norm computation
(as in \cite{Jarlebring:2011:H2})
we need to solve the delay Lyapunov
equation with $W=C_0^TC_0$. 
%Due to the controllability
%and observability duality, we could also
%compute the $\mathcal{H}_2$-norm with the 
%delay Lyapunov equation corresponding to $W=B_0B_0^T$. 
%No substantial difference in computational performance
%was observed and for brevity we only report the result for  $W=C_0^TC_0$.

We carried out simulations of this system using a computer with
an Intel i7 quad-core
processor with 2.1GHz and 16 GB of RAM.
For the finest discretization that we could treat with our
approach, we have $n_x=n_y=23$, $n=1058$, $\|A_0\|_2\approx 5000$
and $\|A_1\|\approx 100$. We  again select $c=1$.

%The PDE has 
%Illustration of the competitiveness in relation to other 
%approaches.

%Somewhat problematic: Tsylv-solving dominates the computation cost and does not appear to be very effective 
%preconditioner since the methods converge very slowly (or have a long intermediate almost-stagnation phase).
%

%scaling: We can introduce $\tilde{U}(t)=U(\alpha t)$, 
%which corresponds to the transformation
%$\tilde{A}_0=\alpha A_0$, 
%$\tilde{A}_1=\alpha A_1$, 
%$\tilde{\tau}=\tau/\alpha$. In order to avoid overflow
%or underflow in $\exp(t A_0)$ we select $\alpha=1/\|A_0\|$.
%

In order to solve the ODE \eqref{eq:Z1Z2} we 
used either a fixed fourth order Runge-Kutta method with $N=500$ grid points, paired with GMRES with tolerance $10^{-8}$,
or the Prince-Dormand method (Matlab's \texttt{ode45}) with adaptive step-size, paired with inexact Krylov with tolerance $10^{-8}$. 
The iteration history of the two variants 
is visualized in Figure~\ref{fig:experiment_large_convergence} 
for $n=1058$. We observe linear convergence and
no substantial difference in convergence rate. 

The execution time of our approach in relation
to some other approaches in the literature is reported in 
Table~\ref{tbl:performance}. Note that these other approaches fail for the larger problems, due to their higher memory requirements.
Discr.~first represents the approach 
discussed in \cite{Vanbiervliet:2011:H2DISC} 
and used in \cite{Jarlebring:2013:BALANCING} with $N=10$ grid points.
This method produces an approximation $\tilde{U}_0$ of $U(0)$, but we do not know
of a simple way to produce an approximation of $U(\tau)$ with it; hence we cannot evaluate the residual measure. We note, however, that this method produces an approximation $\tilde{U}_0$ which differs significantly from the approximation $\widehat{U}_0$ produced by the matrix exponential method.

Note also in Table~\ref{tbl:performance} that the number of iterations required to reach
a specified tolerance appears not to grow substantially 
with the size of problem. Hence, the method
appears to have essentially grid-independent convergence
rate, which is considered a very important
feature of a preconditioner.

Table~\ref{tbl:performance} shows that the inexact method gives results of comparable accuracy in a slightly lower time.

In a detailed profiling of our approach, we identify that
two components are dominating, solving the ODE, i.e., 
computing the action, and solving the T-Sylvester equation. 
For the finest discretization, solving
one T-Sylvester equation took approximately 320 seconds 
and carrying out one step of RK4 required 30 seconds.
We note that the implementation that we have used to solve T-Sylvester equations is not particularly optimized; it is a vectorized version of the algorithm in~\cite{DeTeran:2011:TSYLV} that we have implemented in Matlab for use in these experiments. The complexity in flops of the required computations is only slightly larger than what is required for solving a standard Sylvester equation with the Bartels-Stewart algorithm, a task which requires less than 8 seconds on our machine. Hence, we expect a major reduction in the timings (and a greater difference between the exact and inexact approach) if a carefully optimized solver for the T-Sylvester is used instead.
We also wish to point out that although our theory provides some insight on when the
iterative method is expected to work well, its behavior is still problem dependent. In
Figure~\ref{fig:pdevariantconv} we see that the a different choice of $f_0$ leads
to much faster convergence.

To our knowledge, the largest delay Lyapunov equation
previously solved in literature is with $n=110$ in \cite{Jarlebring:2013:BALANCING}.

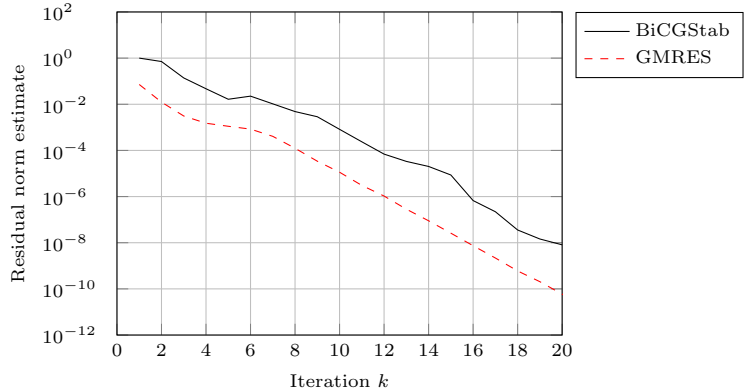
\begin{figure}[h]
  \begin{center}
\tikzsetnextfilename{experiment_large2_convergence} % name for the temporary generated file by tikz externalization in gfx_tmp (see preamble)
\begin{tikzpicture}
\begin{semilogyaxis}%
[width=7.5cm,% 6cm normally enough space to use as two subfigures
 height=5.88cm,% 6cm normally enough space to use as two subfigures 
 xlabel={Iteration $k$},%
 ylabel={Residual norm estimate},%
 % y label style={yshift=-3pt},% y label too close to axis?
 % xtick=data, ytick=data, % if you want data to determine ticks
 xmin=0,xmax=20,%
 ymin=1e-12,ymax=1e2,%
 y tick label style={text width=\widthof{$10^{-10}$}}, %align the y-ticks to the ``left''
 xtick={0,2,4,...,20},%
  ytick={1e-12,1e-10,1e-8,1e-6,1e-4,1e-2,1,1e2},
 legend entries={
BiCGStab,
GMRES
},%
 %legend pos=north west, legend style={draw,cells={anchor=west},row sep=-3pt},%
 legend pos=outer north east, %legend style={draw=none,fill=white,cells={anchor=west},row sep=-2pt,font=0.1},%
 %mark repeat={3},% skip drawing the mark every X timme
 legend cell align=left,
 grid,
 domain=0:100,
 %clip mode=individual,% ?
]
%\addplot[blue,solid] table[x index=0,y index=1]{gfx/convergence_semilogy5.txt};
%\addplot[blue,dashed] table[x index=0,y index=1]{gfx/convergence_semilogy4.txt};
%\addplot[blue,dashdotted] table[x index=0,y index=1]{gfx/convergence_semilogy3.txt};
%% BICGSTAB
\addplot[black, solid,mark=none, mark options={solid,scale=1}] table[x index=0,y index=1]{gfx/experiment_large2_convergence1.txt};
%% GMRES
\addplot[red, dashed,mark=none, mark options={solid,scale=1}] table[x index=0,y index=1]{gfx/experiment_large2_convergence2.txt};
\end{semilogyaxis}
\end{tikzpicture}%}%
    \caption{
Convergence of the iterative methods
with $T$-Sylvester preconditioning
 corresponding to the time-delay system stemming from the discretization of the
PDDE \eqref{eq:pdde} with $n=2n_xn_y=1058$ for
the example in Section~\ref{sect:largeexample}.
      \label{fig:experiment_large_convergence}
    }
  \end{center}
\end{figure}

\begin{table}[h]
  \begin{center}\small
          \begin{tabular}{l||c|c|c|c|c|c}
                    & Matrix exp. \cite{Plischke:2005:TRANSIENT} &Discr. first    & \multicolumn{2}{|c}{RK4 + GMRES} & \multicolumn{2}{|c}{RK45 + inexact GMRES} \\
\hline       $n$& Wall time &Wall time 
&Wall time & iterations &Wall time & iterations \\
            \hline
            \hline
            $28$   & 1.00 sec   & 0.07 sec & 1.15 sec   & 13 & 2.40 sec& 13\\
            $50$   & 141 sec & 0.33 sec  & 3.9 sec  & 15 & 0.74 sec& 14\\
            $242$  & MEMERR  & 111 sec & 116 sec  &17 & 60 sec & 15\\
            $722$  & MEMERR  & MEMERR & 35.6 min &18 & 26.9 min & 16\\
            $1058$ & MEMERR  & MEMERR & 1.79 hrs& 18 & 1.67 hrs & 16
          \end{tabular}

          \medskip

\begin{tabular}{l||c|c|c|c}
                    & Matrix exp. \cite{Plischke:2005:TRANSIENT} &Discr. first    & RK4 + GMRES & RK45 + inexact GMRES\\ \hline
                    $n$ & $\operatorname{res}(\tilde{U}_0,\tilde{U}_\tau)$ & $\operatorname{res}(\tilde{U}_0,\tilde{U}_\tau)$ & $\operatorname{res}(\tilde{U}_0,\tilde{U}_\tau)$ & $\operatorname{res}(\tilde{U}_0,\tilde{U}_\tau)$\\
                    \hline \hline
            $28$ & $1.4\times 10^{-13}$ & N/A & $1.6\times 10^{-8}$   & $1.7\times 10^{-8}$\\
            $50$ & $1.7\times 10^{-11}$ & N/A & $6.2\times 10^{-9}$ & $2.7\times 10^{-8}$\\
            $242$  & MEMERR & N/A & $1.6\times 10^{-8}$   & $1.7\times 10^{-7}$\\
            $722$  & MEMERR & MEMERR & $2.2\times 10^{-8}$   & $1.8\times 10^{-7}$\\
            $1058$ & MEMERR & MEMERR & $3.8\times 10^{-8}$   & $2.5\times 10^{-7}$
\end{tabular}

\medskip

\begin{tabular}{l||c|c|c|c}
                    & Matrix exp. \cite{Plischke:2005:TRANSIENT} &Discr. first    & RK4 + GMRES & RK45 + inexact GMRES\\ \hline
                    $n$ & $\frac{\norm{\tilde{U}_0 -\widehat{U}_0}}{\norm{\widehat{U}_0}}$ & $\frac{\norm{\tilde{U}_0 -\widehat{U}_0}}{\norm{\widehat{U}_0}}$ & $\frac{\norm{\tilde{U}_0 -\widehat{U}_0}}{\norm{\widehat{U}_0}}$ & $\frac{\norm{\tilde{U}_0 -\widehat{U}_0}}{\norm{\widehat{U}_0}}$\\
                    \hline \hline
            $28$ & $0$ & $2.7\times 10^{-4}$ & $6.7\times 10^{-9}$   & $7.7\times 10^{-9}$\\
            $50\phantom{18}$ & $0$ & $1.8\times 10^{-2}$ & $2.4\times 10^{-9}$ & $1.1\times 10^{-8}$
\end{tabular}
    \caption{
      \label{tbl:performance}
Performance in relation to other methods: time, iterations residual, error in $\tilde{U}_0$ with respect to the Matrix exp.~method. 
    }
  \end{center}
\end{table}

%
%\begin{figure}[h]
%  \begin{center}
%    %%Traditional figure import with PDFs:
%    %\subfigure[mytext]{\scalebox{0.65}{\includegraphics{gfx/file.pdf}}}
%    %\scalebox{0.7}{\includegraphics{gfx/file.pdf}}
%    %
%    %% Fancy-pantsy tikz figure:
%    \subfigure[Convergence for the]{\input{gfx/convergence_fig1.tex}}%
%    \subfigure[Timing]{
%          \begin{tabular}{r|c|c|c}
%            ${\rm TOL}_0$& $10^{-6}$ &$10^{-8}$  & $10^{-10}$ \\
%            \hline
%            CPU-time &2.8 s & 3.6 s & 6.2 s
%          \end{tabular}
%}
%    %\subfigure[hello]{\input{gfx/myfig1.tex}}%
%    %\subfigure[hello]{\input{gfx/myfig1.tex}}%
%    \caption{
%      Convergence for GMRES with inexact matrix
%vector products with tuned tolerance
%      %\label{fig:}
%    }
%  \end{center}
%\end{figure}
%
%%
%\begin{figure}[h]
%  \begin{center}
%    %\subfigure[mytext]{\scalebox{0.65}{\includegraphics{file}}}
%%    \scalebox{0.7}{\includegraphics{}}
%    \caption{Figure illustrating
%how well the preconditioner works.
%      %\label{fig:}
%    }
%  \end{center}
%\end{figure}
%

\begin{figure}[h]
  \begin{center}
\tikzsetnextfilename{experiment_small_convergence} % name for the temporary generated file by tikz externalization in gfx_tmp (see preamble)
\begin{tikzpicture}
\begin{semilogyaxis}%
[width=7.5cm,% 6cm normally enough space to use as two subfigures
 height=5.88cm,% 6cm normally enough space to use as two subfigures 
 xlabel={Iteration $k$},%
 ylabel={Residual norm},%
 % y label style={yshift=-3pt},% y label too close to axis?
 % xtick=data, ytick=data, % if you want data to determine ticks
 xmin=0,xmax=16,%
 ymin=1e-10,ymax=1e0,%
 y tick label style={text width=\widthof{$10^{-10}$}}, %align the y-ticks to the ``left''
 xtick={0,2,4,...,16},%
 ytick={1e-14,1e-12,1e-10,1e-8,1e-6,1e-4,1e-2,1,1e2},
 legend entries={
$f_0=0.5$,
$f_0=5$
},%
 %legend pos=north west, legend style={draw,cells={anchor=west},row sep=-3pt},%
 legend pos=outer north east, %legend style={draw=none,fill=white,cells={anchor=west},row sep=-2pt,font=0.1},%
 %mark repeat={3},% skip drawing the mark every X timme
 legend cell align=left,
 grid,
 domain=0:100,
 %clip mode=individual,% ?
]
%\addplot[blue,solid] table[x index=0,y index=1]{gfx/convergence_semilogy5.txt};
%\addplot[blue,dashed] table[x index=0,y index=1]{gfx/convergence_semilogy4.txt};
%\addplot[blue,dashdotted] table[x index=0,y index=1]{gfx/convergence_semilogy3.txt};
%% BICGSTAB
\addplot[black, solid,mark=x, mark options={solid,scale=1}] table[x index=0,y index=1]{gfx/pdevariantconv1.txt};
\addplot[black, solid,mark=o, mark options={solid,scale=1}] table[x index=0,y index=1]{gfx/pdevariantconv2.txt};
%\addplot[black, solid,mark=square, mark options={solid,scale=1}] table[x index=0,y index=1]{gfx/experiment_small_convergence9.txt};
%\addplot[black, solid,mark=diamond, mark options={solid,scale=1}] table[x index=0,y index=1]{gfx/experiment_small_convergence6.txt};
%\addplot[black, solid,mark=*, mark options={solid,scale=1}] table[x index=0,y index=1]{gfx/experiment_small_convergence3.txt};
%% GMRES
%\addplot[red, dashed,mark=x, mark options={solid,scale=1}] table[x index=0,y index=1]{gfx/experiment_small_convergence14.txt};

\end{semilogyaxis}
\end{tikzpicture}
    \caption{The convergence of GMRES for different choices of $f_0$. 
      \label{fig:pdevariantconv}
    }
  \end{center}
\end{figure}
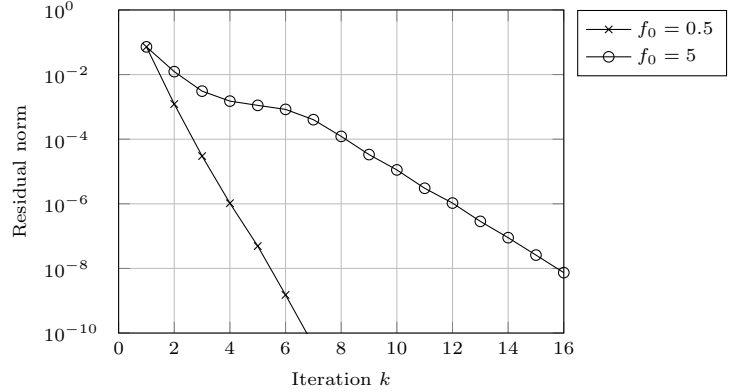

\section{Concluding remarks and outlook}\label{sect:conclusions}

We have in this paper proposed a procedure
to solve delay Lyapunov equations
based on iterative methods for linear systems
combined with a direct method for T-Sylvester equations. 
% which scales better
%with the dimension of the problem, in relation
%to other methods. 
Although the method
performs well in practice, there appears to be 
possibilities to improve it further, which
we consider beyond the scope of the paper.  
%Due to the fact that our
%approach involves  solving a linear system
%with using an iterative method, there
%are various ways to improve iterative methods. 

As observed in the simulations, the dominating ingredient
of the approach is the solution to the T-Sylvester equation.
Hence, in order to solve even larger problems we need new 
methods for $T$-Sylvester equations. Improvements
are possible, e.g., by lower level implementations,
or by developing methods which can  take the sparsity of the matrices into account, e.g., similar to the 
Krylov methods and rational Krylov methods
for Lyapunov equations \cite{Simoncini:2007:KPIK}
or approaches based on Riemannian optimization
\cite{Vandereycken:2010:LYAP}.

Our work on inexact Krylov methods may also allow extension to other
types of iterative methods, in particular flexible variants of GMRES
\cite{Saad1993}. Although the flexible variants of GMRES can work
better in situations where the preconditioner changes in every iteration,
the understanding of their convergence is less mature.

The preconditioner in general plays an important
role in iterative methods for linear systems
and the effectiveness of the preconditioner is typically
very
problem-dependent. This is also the case in our
approach. Although the simulations often worked well,
during some experiments, 
in particular situations where $A_0$ have some eigenvalues which are
very negative, the preconditioner did not appear very effective,
even if $\|A_1\|$ was quite small.
This can be due to the fact that the hidden
constant in the
expression \eqref{eq:precondquality} may be large.

The delay Lyapunov equation has been generalized
in several ways, e.g., to multiple delays
and neutral systems. Our approach might be generalizable 
to some of these cases. The simplest situations appears
to be if the delays are integer
multiplies of each other, also known
as commensurate delays.
For the commensurate case there
are procedures which resemble our reformulation \eqref{eq:Z1Z2}
with Sylvester resultant matrices
\cite[Problem~6.72]{Plischke:2005:TRANSIENT}.
However, this increases the size of the problem.
An attractive feature of our approach is 
that we work only with matrices of size $n$, which
would not be the case in the direct adaption to multiple
commensurate delays using \cite[Problem~6.72]{Plischke:2005:TRANSIENT}.

\section*{Acknowledgments}
The authors thank Antti Koskela and Tobias Damm
for discussions about early results of the paper. 

F.~Poloni acknowledges the support of the PRA 2014 project ``Mathematical models and computational methods for complex networks'' of the University of Pisa, and of INDAM (Istituto Nazionale di Alta Matematica). E.~Jarlebring acknowledges the support of the Swedish research council (Vetenskapsr\aa{}det) project 2013-4640.

We thank the referees and editor for their constructive comments.

\def\appendixname{Appendix } %Federico was here
\appendix
\section{Technical bounds} \label{sec:appendixtechnical}
\noindent
The following results are needed in the proof
of Theorem~\ref{thm:precondquality}.
\begin{lemma}\label{lem:Zbound}
Suppose $Z_1$ and $Z_2$ satisfy \eqref{eq:Z1Z2} with 
initial condition $Z_1(0)=Z_2(0)=X$. For $i=1,2$, 
\[
 \|Z_i(t)\|_F \le 2\exp(2t(\norm{A_0}+\norm{A_1}))\|X\|_F.
\]
\end{lemma}
\begin{proof} 
We rely on the vectorized form~\eqref{eq:vectorizedODE} of the ODE defining $Z_i(t)$; we have
\begin{equation*}
\norm{Z_i(t)}_F \le \norm*{\begin{bmatrix}\vec Z_1(t)\phantom{^T}\\ \vec Z_2(t)^T\end{bmatrix}} \le
\norm{\exp(t\mathcal{A})}\norm*{\begin{bmatrix}\vec X\phantom{^T}\\\vec X^T \end{bmatrix}} \leq 2\exp(t\norm{\mathcal{A}})\norm{X}_F.
\end{equation*}
To complete the proof, we have to estimate the norm of the matrix $\mathcal{A}$ in~\eqref{eq:defmathcalA}: we have
\begin{multline*}
\norm{\mathcal{A}} \leq \norm{A_0^T\otimes I_n} + \norm{A_1^T\otimes I_n} + \norm{I_n\otimes A_1^T} + \norm{I_n\otimes A_0^T} =\\ 2(\norm{A_0}+\norm{A_1}),
\end{multline*}
where we have used the fact that $\norm{M\otimes N} = \norm{M}\norm{N}$.
\end{proof}
\begin{lemma}\label{thm:tLc_bound}
Suppose that $A_0$ has no Hamiltonian eigenpairing. 
Then, there exists a constant $K$ depending only on $A_0$ and $c$ such that
\[
 \|\tilde{L}_c^{-1}(Z)\|_F\le K \exp(\tau\|A_0\|/2)\|Z\|_F.
\]
\end{lemma}
\begin{proof}
Under the stated hypotheses, $T$ is invertible. Let $K$ be the operator norm of $T^{-1}$, i.e., the smallest constant such that $\|T^{-1}(Z)\|_F \leq K \|Z\|_F$. Then 	
\begin{multline}
\|\tilde{L}^{-1}_c(Z)\|_F=\|T^{-1}(Z)\exp(\tau A_0/2)\|_F \leq \\\|T^{-1}(Z)\|_F\|\exp(\tau A_0/2)\| \leq K \|Z\|_F \exp(\tau \|A_0\|/2),
\end{multline}
where we have used the mixed matrix norm inequality $\norm{MN}_F\leq \norm{M}_F\norm{N}$ \cite[Page~50-5, Fact~10]{Hogben:2014:Handbook}.
\end{proof}

\section*{References}
\bibliography{eliasbib,misc}

\end{document}